\documentclass[12pt]{amsart}

\usepackage{amssymb, amsmath}

\def\sA{\langle A\rangle}
\def\sB{\langle B\rangle}
\def\sC{\langle C\rangle}

\def\rA{k[A]}
\def\rB{k[B]}
\def\rC{k[C]}

\def\C{\mathfrak{C}}
\def\N{\mathbb{N}}
\def\Z{\mathbb{Z}}
\def\A{\mathbb{A}}
\newcommand{\reg}{\mathrm {reg }}
\newcommand{\pd}{\mathrm {pd }}

\def\rar{\rightarrow}
\def\lar{\longrightarrow}
\def\coker{\operatorname{coker}}

\newtheorem{theorem}{Theorem}[section]
\newtheorem{lemma}[theorem]{Lemma}
\newtheorem{corollary}[theorem]{Corollary}
\newtheorem{proposition}[theorem]{Proposition}

\theoremstyle{definition}

\newtheorem{remark}{Remark}
\newtheorem{example}{Example}

\title[Minimal free resolution of semigroup rings obtained by gluing]{The structure of the minimal free resolution of semigroup rings obtained by gluing}

\author{Philippe Gimenez}
\address{Instituto de Matem\'aticas de la Universidad de Valladolid (IMUVA),
Facultad de Ciencias,
47011 Valladolid, Spain.}
\email{pgimenez@agt.uva.es}
\thanks{The first author was partially supported by {\it Ministerio de Ciencia e Innovaci\'on} (Spain), MTM2016-78881-P}

\author{Hema Srinivasan}
\address{Mathematics Department, University of Missouri, Columbia, MO 65211, USA.}
\email{SrinivasanH@math.missouri.edu}



\begin{document}
\maketitle

\begin{abstract} 
We construct a minimal free resolution of the semigroup ring $\rC$ in terms of minimal resolutions of $\rA$ and $\rB$
when $\sC$ is a numerical semigroup obtained by gluing two numerical semigroups $\sA$ and $\sB$.
Using our explicit construction, we compute the Betti numbers, graded Betti numbers, regularity and Hilbert series of $\rC$, and
prove that the minimal free resolution of $\rC$ has a differential graded algebra structure provided the resolutions of $\rA$ and $\rB$ possess them.
We  discuss the consequences of our results in small embedding dimensions.  Finally, we give an extension of our main result to semigroups in $\N^n$.
\end{abstract}

\section{Introduction}
\label{introSec}

Given a subset of positive integers $C= \{c_1, \ldots, c_n\}$, we denote by $\sC$ the semigroup generated by $C$.
When $\gcd(c_1,\ldots,c_n)=1$, i.e., when $\N\setminus\sC$ is a finite set, we call $\sC$ a {\it numerical} semigroup.
We say $C$ minimally generates the semigroup $\sC$ if no proper subset of $C$ generates the same semigroup $\sC$.
The minimal generating set of a semigroup is unique.

\medskip
Let $k$ be an arbitrary field.
Consider the algebra homomorphism $\phi _C: k[x_1, \ldots,x_n]\to k[t]$ induced by $\phi_C(x_i) = t^{c_i}$.
The semigroup ring $\rC =k[t^{c_i}|1\le i\le n]$ is the image of $\phi_C $ and it is isomorphic to $k[x_1, \ldots, x_n]/I_C$ where $I_C = \ker \phi_C$ is the prime ideal defining the affine monomial curve  $\C_C\subset\A_k^n$ whose coordinate ring is $\rC$.
The ideal $I_C$ is binomial and it is homogeneous if one gives to each variable $x_i$ weight $c_i$. Hence, one can consider graded free resolutions
of $\rC$ over the polynomial ring graded this way.
If $\sC$ is minimally generated by $C$, we say that the {\it embedding dimension} of $\sC $ is $n$.
The ring $\rC$ has Krull dimension $1$ and hence it is always Cohen-Macaulay.
Observe that if one studies the semigroup ring $\rC$, one can always assume without loss of generality that $\gcd(c_1,\ldots,c_n)=1$ since
one gets the same ideal $I_C$ when one divides all the integers in $C$ by their gcd.
When the ideal $I_C$ is generated by a regular sequence, we say the numerical semigroup
$\sC$, the semigroup ring $\rC$, and the monomial curve $\C_C$, are complete intersections.  Similarly, we say the semigroup $\sC$ and
the monomial curve $\C_C$ are
Gorenstein when the semigroup ring $\rC$ satisfies this property. Note that the fact that $\rC$ is Gorenstein or a complete intersection
does not depend on the field $k$ and that's why we can talk about Gorenstein and complete intersection semigroups.

\medskip
Numerical semigroups of embedding dimension 1 are isomorphic to the integers and hence their semigroup ring is simply the polynomial ring $k[x]$.
When the embedding dimension is $2$, the semigroup ring $\rC\simeq k[x,y]/(x^a-y^b)$ is the coordinate ring of a hypersurface.
When the embedding dimension is three,
$\rC\simeq k[x,y,z]/I_C$ is either a complete intersection or $I_C$ is the ideal of the $2\times 2$ minors of a $2\times 3$ matrix of the form
$\displaystyle{
\left(\begin{array}{ccc}
x^a&y^b&z^c\\ y^d&z^e&x^f
\end{array}\right)
}$ by \cite[Theorems 3.7 and 3.8]{He}.
When the embedding dimension is four or more, there is no such simple classification for the semigroup rings.
In fact, when $C = \{c_1,c_2,c_3,c_4\}$, there is no upper bound for  the minimal number of generators for the ideal $I_C$, \cite {Br}.

\medskip
In this paper, we construct a minimal graded free resolution for the semigroup ring $\rC$ when $\sC$ is a numerical semigroup
obtained by gluing two smaller numerical semigroups.
This is our main Theorem \ref{main}.

\medskip
Our paper is organized as follows.
In section \ref{gluingSec}, we recall the definition of semigroups obtained by gluing and
give some of their properties, known and new.
Section \ref{mainSec} has our main result and the consequences.
For numerical semigroups obtained by gluing, we compute the invariants associated to the minimal graded free resolution like the Betti numbers (Corollary \ref{BettiNumbers}), the graded Betti numbers (Proposition \ref{gradedBettiNumbers}), the Castelnuovo-Mumford regularity (Corollary \ref{regularity}) and the Hilbert series (Corollary \ref{Hilbert series}).
In particular, we show that if $\sC$ is obtained by gluing $\sA$ and $\sB$,
then the Cohen-Macaulay type of $\rC$ is the product of the Cohen-Macaulay types of $\rA$ and $\rB$ (Corollary \ref{CMtype}). We also show that the minimal resolution of $\rC$ has a differential graded algebra structure provided the minimal resolutions of $\rA$ and $\rB$ do (Proposition \ref{DGstructure}).
The multiplication on the resolution of $\rC$ is explicitly written from those of $\rA$ and $\rB$.
In section \ref{simplesplitSec}, we consider the case when the numerical semigroup $\sC$ decomposes as $C = k_1A\sqcup k_2B$ where one of the subset of $C$ is a singleton.
In section \ref{smallembdimSec}, we discuss the consequences of our results for small embedding dimensions and give some examples.
In particular, we classify decomposable numerical semigroups up to embedding dimension 5 (Propositions \ref{embdim4} and \ref{embdim5})
and also show that, in embedding dimension $n\geq 4$, there exist indecomposable numerical semigroups of any Cohen-Macaulay type
between 1 and $n$ (Corollary \ref{indecompCMtype}).
Finally, in the last section we observe that our main result is valid for decomposable semigroups in higher dimension (Theorem \ref{mainHD})
and give an example to illustrate this.

\section{Numerical semigroups obtained by gluing}
\label{gluingSec}

\subsection{Definition}\label{defGluingSec}

A numerical semigroup $\sC$ is obtained by {\it gluing} two semigroups $\sA$ and $\sB$ if its minimal generating set $C$ can be written as
the disjoint union of two subsets, $C=k_1A\sqcup k_2 B$,
where $A$ and $B$
are minimal generators of numerical semigroups $\sA$ and $\sB$, and
$k_1$, $k_2$ are relatively prime positive integers such that $k_1$ is in the numerical semigroup $\sB$ but not in $B$ and $k_2\in\sA\setminus A$.
When this occurs, we also say that $\sC$ is {\it decomposable},
or that $\sC$ is a {\it gluing of $\sA$ and $\sB$}.

\medskip
This way to decompose the minimal generating set of a numerical semigroup
already appears in the classical paper by Delorme \cite{De} where it is used to characterize complete intersection numerical semigroups. In \cite{rosales}, Rosales introduces the concept of gluing for finitely generated subsemigroups of $\N^n$ in terms
of a presentation for the semigroup, and he shows that, for numerical semigroups, his definition coincides with the decomposition in \cite{De};
see section \ref{propertiesGluingSec} below.

\medskip
Note that one can extend the above definition when $k_1$ and $k_2$ are not relatively prime.
If the positive integers in $C$ are not relatively prime and $d$ is their gcd, then the semigroup $\sC$ is {\it decomposable} if and only if the numerical semigroup $\langle \frac{C}{d}\rangle$ is decomposable.  In other words, if $C=k_1A\sqcup k_2B$ and
$k_i= dk_i^\ast$ for $i= 1, 2$ with $d=\gcd(k_1,k_2)$, we ask
that $k_1^\ast\in \sB\setminus B$ and $k_2^\ast\in \sA\setminus A$ to say that $\sC$ is a gluing of  $\sA$ and $\sB$.

\subsection{Notations}\label{notationsGluingSec}

When $C=k_1A\cup k_2 B$ with $k_1\in\sB$ and $k_2\in\sA$, in particular
if $\sC$ is obtained by gluing $\sA$ and $\sB$, we will use the following notations:
\begin{itemize}
\item
$p=\# A$, $q=\# B$, and
the elements in $A$ and $B$ are $\{a_1, \ldots, a_p\}$ and $\{ b_1, \dots, b_q\}$.
\item
We consider the three polynomial rings $R_A=k[x_1, \ldots, x_p]$, $R_B=k[y_1, \dots, y_q]$ and
$R= k[x_1, \ldots, x_p, y_1,\ldots, y_q]$.
Then, the semigroup ring $\rA$ is isomorphic to $ R_A/I_A$ where $I_A$ is the kernel of the algebra homomorphism
$\phi _A: R_A\to k[t]$ induced by $\phi_A(x_i) = t^{a_i}$.
Similarly, $\rB\simeq R_B/I_B$ and $\rC\simeq R/I_C$.
\item
$F_A$ and $F_B$ are minimal graded free resolutions for $\rA$ and $\rB$ as modules over $R_A$ and $R_B$ respectively,
where variables $x_i$ and $y_j$ in the polynomial rings $R_A$ and $R_B$ are given weight $a_i$ and $b_j$ to make the modules
$\rA$ and $\rB$ graded.
When needed, the differentials will be denoted by
$\partial_i^{A}:(F_A)_i\rar (F_A)_{i-1}$ and $\partial_j^{B}:(F_B)_j\rar (F_B)_{j-1}$ for $1\leq i\leq p-1$ and $1\leq j\leq q-1$.
By definition, the resolutions $F_A$ and $F_B$ are {\it minimal} if and only if
$\partial_i^{A}({(F_A)}_i)\subset (x_1, \ldots, x_p){(F_A)}_{i-1}$ and
$\partial_j^{B}({(F_B)}_j)\subset  (y_1, \ldots, y_q){(F_B)}_{j-1}$ for all $1\leq i\leq p-1$ and $1\leq j\leq q-1$.
\item
Since $k_1\in\sB$ and $k_2\in\sA$, there exist non negative integers $\alpha_i , \beta_i$ such that
$k_1 = \sum_{j=1}^{q} \beta_j b_j$ and $k_2 = \sum_{i=1}^{p} \alpha_i a_i$, and we consider
the binomial $\rho = \prod_{i=1}^{p} x_i^{\alpha_i}- \prod_{j=1}^{q} y_j^{\beta_j} \in R$. It is homogeneous of degree $k_1k_2$
if one gives to each variable in $R$ the corresponding weight in $C=\{k_1a_1,\ldots,k_1a_p,k_2b_1,\ldots,k_2b_q\}$.
\end{itemize}

\subsection{Properties}\label{propertiesGluingSec}

The following lemma helps to understand why, in the definition of gluing, we consider $k_1\notin B$ and $k_2\notin A$.
It is slightly more precise than \cite[Prop. 10 (ii)]{De}.

\begin{lemma}
If $A$ and $B$ minimally generate $\sA$ and $\sB$ respectively, and $k_1\in \sB$ and $k_2\in \sA$ are two relatively prime positive integers,
then $C= k_1A\cup k_2B$ is a disjoint union and minimally generates the semigroup $\sC$ if and only if
$k_1\notin B$ and $k_2\notin A$.
\end{lemma}

\begin{proof}
First observe that the assumption in the lemma is necessary: if $k_1\in B$
or $k_2\in A$ then $k_1k_2\in C$ and it is not a minimal generator of $\sC$ unless $k_1\in B$
and $k_2\in A$, but in this case $k_1k_2\in k_1A$ and $k_1k_2\in k_2B$ and hence the union $k_1A\cup k_2B$ is not disjoint.

\medskip
Now assume that $k_1\notin B$ and $k_2\notin A$.
If $C$ is not a minimal generating set of $\sC$ or if the union $C= k_1A\cup k_2B$ is not disjoint,
we may say without loss of generality that $k_2b_q $ is in the numerical semigroup generated by the rest of the elements in $C$.
That is, there are non negative integers $c_i$ for $1\leq i \leq p$ and $d_j$ for $1\leq j\leq q-1$ such that
$$
k_2b_q = \sum_{i=1}^{p} c_ik_1a_i + \sum_{j=1}^{q-1}d_jk_2 b_j\,.
$$
Hence, $k_2$ divides $k_1 \sum_{i=1}^{p} c_ia_i$ and, since $(k_1,k_2) = 1$, one has that $\sum_{i=1}^{p} c_ia_i = k_2t$ for some
non negative integer $t$.  Moreover, $t$ must be positive or else $B$ will not be a minimal generating set for $\sB$.

\medskip
Dividing the above equation by $k_2$, we get
$$
b_q = k_1t+ \sum_{j=1}^{q-1} d_jb_j\,.
$$

Now, note that $k_1 = \sum_{j=1}^{q} r_j b_j$ for some non negative integers $r_j$.
Thus, we get
$$
\sum_{j=1}^{q-1} (tr_j+d_j) b_j + (tr_q-1) b_q = 0\,.
$$
Since $tr_j+d_j \ge 0$ and $tr_q\ge 0$, this can only occur if either $d_j=r_j=0$ for all $j=1,\ldots, q-1$ and $t=r_q=1$, or $tr_q = 0$.
In the first case, one has that $b_q=k_1t=k_1$ which is impossible if $k_1\notin B$.
On the other hand, $tr_q = 0$ implies $r_q = 0$.
But then,
$b_q = \sum _{j=1}^{q-1} (tr_j+d_j)b_j$ which contradicts the fact that  $B$ is a minimal generating set for $\sB$ .
Hence, $C=k_1A\sqcup k_2B$ minimally generates the semigroup $\sC$ and we are done.
\end{proof}

For a numerical semigroup, the definition that we gave in section \ref{defGluingSec} is equivalent to Rosales' original definition of gluing and this is proved in \cite[Lemma 2.2]{rosales}.
We add our proof of the following result for completion and convenience:

\begin{proposition}[{\cite[Lemma 2.2]{rosales}}]\label{generatorsGluing}
If $\sC$ is obtained by gluing $\sA$ and $\sB$, then the ideal $I_C$ is minimally generated
by the minimal generators of $I_A$ and $I_B$ and $\rho$, the binomial defined in
section \ref{notationsGluingSec}.
\end{proposition}

\begin{proof}
First note that $\sum_{i=1}^p c_i(k_1a_i)+\sum_{j=1}^q d_j(k_2b_j )= 0$ for some integers $c_i, d_j$
is a minimal relation on the generators of the semigroup $\sC$
if and only if
$\prod _{c_i>0}x_i^{c_i}\prod_{d_j>0}y_j^{d_j} -\prod _{c_i<0}x_i^{-c_i}\prod_{d_j<0}y_j^{-d_j}$ is a minimal generator
for the ideal $I_C$.

\medskip
With notations in section \ref{notationsGluingSec},
since
\begin{equation}\label{rhoEq}
\sum_{i=1}^{p} \alpha_i (k_1 a_i) - \sum_{j=1}^{q} \beta_j (k_2 b_j)=0
\end{equation}
for non negative integers $\alpha_i$ and $\beta_j$, one has that $\rho\in I_C\setminus (I_A+I_B)$.

\medskip
Let $\sigma =\sum_{i=1}^{p} c_i (k_1a_i)+\sum_{j=1}^{q} d_j(k_2b_j) = 0$ be a relation on the elements in $C$ for some integers $c_i$ and $d_j$.  Then,
$k_1 ( \sum_{i=1}^{p} c_ia_i) = -k_2( \sum_{j=1}^{q} d_jb_j)$. Since $k_1$ and $k_2$ are relatively prime, we conclude that
$\sum_{i=1}^{p} c_i a_i = k_2 r= r\sum_{i=1}^{p}\alpha_i a_i$ and $\sum_{j=1}^{q} d_jb_j =- k_1r= -r \sum_{j=1}^{q} \beta_j b_j$.
So, $\sum_{i=1}^{p}(c_i-r\alpha_i) a_i = 0 = \sum_{j=1}^{q} (d_j+r\beta_j) b_j $.  Hence,
$\sigma = \sum_{i=1}^{p}(c_i-r\alpha_i) k_1 a_i +  \sum_{j=1}^{q}(d_j+r\beta_j) k_2 b_j $.
Thus, any relation $\sigma$  among the integers in $C$ is a sum of relations among elements in $A$ and relations among elements $B$,
unless $c_i =r\alpha_i$ and $d_j = -rb_j$ for all $i,j$.
In that case, $\sigma$ is a multiple of the relation (\ref{rhoEq}) and hence $I_C\subseteq I_AR+I_BR+\langle\rho\rangle$. The other inclusion is trivial,
and it is clear that taking minimal generators of $I_A$ and $I_B$ and adding $\rho$, one gets a sistem of generators of $I_C$ which is minimal.
\end{proof}

\section{Main result and consequences}
\label{mainSec}

With the notations in section \ref{notationsGluingSec}, our main result is the following:

\begin{theorem}\label{main}
Consider a numerical semigroup $\sC$ minimally generated by $C$ and obtained by gluing $\sA$ and $\sB$ with $C = k_1 A\sqcup k_2 B$.
\begin{enumerate}
\item
$F_A\otimes F_B$ is a minimal graded free resolution of $R/I_A R+I_B R$.
\item
A minimal graded free resolution of the semigroup ring $\rC$ can be obtained as the mapping cone
of $\rho: F_A\otimes F_B\to F_A\otimes F_B$, where $\rho$ is induced by multiplication by $\rho$.
In particular, $(I_A R+I_B R:\rho) = I_A R+I_B R$.
\end{enumerate}
\end{theorem}

\begin{proof}
The minimal free resolutions $F_A$ and $F_B$ of $\rA\simeq R_A/I_A$ and $\rB\simeq R_B/I_B$ can be viewed as
minimal free resolutions of the $R$-modules  $R/I_AR$ and $R/I_BR$ (with the same maps).
Since $I_A R$ and $I_B R$ are generated by elements involving disjoint sets of variables in $R$,
$G=F_A\otimes F_B $ is a minimal graded free resolution of  $R/(I_A R+I_B R)$.

\medskip
Now consider the complex map on $G\to G$ induced by multiplication by $\rho$.  All the maps in the complex are also multiplication by $\rho$ and hence this is a complex map.  We note that  $I_AR+I_BR+\rho $ is the prime ideal $I_C$ of height $n-1$.  In the mapping cone $M(\rho)$ of the multiplication map $\rho: G\to G$, all of the fitting ideals contain $I_A R+I_B R+\rho$ up to radical and hence are of depth $n-1$.
So, by the exactness criterion \cite[Cor. 1]{BE1}, the mapping
cone is acyclic and resolves $I_C= I_A R+I_B R+\rho$.

\medskip
Minimality of our resolutions are straightforward.
Since $F_A$ and $F_B$ are minimal,
one has that $\partial_i^{A}((F_A)_i)\subset (x_1, \ldots, x_p)(F_A)_{i-1}$ and
$\partial_j^{B}((F_B)_j)\subset  (y_1, \ldots, y_q)(F_B)_{j-1}$ for $1\leq i\leq p-1$ and $1\leq j\leq q-1$.
Thus, denoting by $\partial^G_\bullet$ the differentials in $G$, one has that
$\partial^{G}_k(G_k)\subset (x_1, \ldots, x_p, y_1, \ldots y_q)G_{k-1}$ for all $k$, $1\leq k\leq p+q-2$, and hence $G$ resolves $R/I_A R+I_B R$ minimally.
Finally, since the maps in the complex map $G\to G$ are all multiplication by $\rho$ and $\rho \in (x_1, \ldots, x_p, y_1\ldots, y_q)$,
we see that if $\partial_\bullet$ are the differentials in the mapping cone $M(\rho)$, for all $k$, $1\leq k\leq p+q-1$, one has
$\partial_{k}(M(\rho)_k)\subset (x_1, \ldots, x_p, y_1, \ldots y_q)M(\rho)_{k-1}$, and $M(\rho)$ resolves $R/I_C$ minimally.

\medskip
As a consequence, we see that the $H_1(M(\rho)) =0$ and hence $(I_A R+I_B R :\rho) = I_A R+I_B R$.
\end{proof}

\begin{example}\label{ex1}
Consider the decomposable numerical semigroup $\sC$ generated by
$$
C=\{187,289,425,323,140,364,336\}=k_1 A\sqcup k_2 B
$$
with $A=\{11, 17, 25, 19\}$, $B=\{5,13,12\}$,  
$k_1=17=5+12\in\sB$ and $k_2=28=11+17\in\sA$. 
Using {\tt Singular} \cite{singular}, one gets that the
ideal $I_A\subset R_A=k[x_1,\ldots,x_4]$ is minimally generated by 5 binomials,
$f_1=x_1x_3-x_2x_4$, $f_2=x_4^4-x_2^3x_3$, $f_3=x_2^4-x_1x_4^3$, $f_4=x_3^2-x_1^3x_2$, $f_5=x_1^4-x_3x_4$,
the ring $\rA\simeq R_A/I_A$ is Gorenstein, and its minimal free resolution is
$$
0\rar R_A
\stackrel{{\tiny
\left(\begin{array}{c}
f_1\\ f_2\\ f_3 \\ f_4\\ f_5
\end{array}\right)}
}{\lar}  (R_A)^5
\stackrel{{\tiny
\left(\begin{array}{rrrrr}
0          & x_3  & x_1^3& x_2^3& x_4^3   \\
-x_3     & 0     & 0        & x_1& x_2           \\
-x_1^3 & 0     & 0        & x_4  & x_3        \\
-x_2^3 & -x_1& -x_4    & 0       & 0               \\
-x_4^3 & -x_2 & -x_3   & 0        & 0
\end{array}\right)}
}{\lar}  (R_A)^5
\stackrel{{\tiny
\left(\begin{array}{c} f_1\ \ldots\  f_5\end{array}\right)}
}{\lar} R_A\rar \rA\rar 0\,.
$$
The ideal $I_B\subset R_B=k[y_1,y_2,y_3]$ is generated by 3 binomials, $g_1=y_3^3- y_1^2y_2^2$,
$g_2=y_1^5-y_2y_3$,
$g_3=y_2^3-y_1^3y_3^2$,
it is Hilbert-Burch, and its minimal free resolution is
$$
0\rar (R_B)^2
\stackrel{{\tiny
\left(\begin{array}{rr}
 y_1^3 & y_2    \\
y_2^2 & y_3^2   \\
y_3 & y_1^2
\end{array}\right)}
}{\lar}
(R_B)^3
\stackrel{{\tiny
\left(\begin{array}{c}g_1\ g_2\  g_3\end{array}\right)}
}{\lar}
R_B\rar \rB\rar 0\,.
$$
The tensor product of these two resolutions provides a minimal free resolution of $R/J$  where $J=I_A R+I_B R$:
$$
0\rar R^2\rar R^{13}\rar R^{26}\rar R^{22}\rar R^8\rar R\rar R/J\rar 0\,.
$$
Note that the differentials can be easily written down if needed. Finally, the extra minimal generator in
$I_C$ is $\rho=x_1x_2-y_1y_3$ and the mapping cone induced by multiplication by $\rho$ gives a minimal resolution of
$\rC$,
$$
0\rar R^2\rar R^{15}\rar R^{39}\rar R^{48}\rar R^{30}\rar R^9\rar R\rar \rC\rar 0\,.
$$
Again, the differentials are known.
\end{example}

Of course, the differentials in the resolution of $\rC$ depend on $k_1$ and $k_2$ when the numerical semigroup decomposes
as $C = k_1A\sqcup k_2B$ but the Betti numbers do not as our next result shows.
For any $R$-module $M$, we denote by $\beta_i(M)$ its Betti numbers with the convention that $\beta_i(M)=0$ for $i<0$ and $i>\hbox{\rm pd}(M)$,
the projective dimension of $M$.
For the semigroup ring $\rC$ associated to a numerical semigroup $\sC$, we will use the notation $\beta_{i}(C)=\beta_{i}(\rC)$.

\begin{corollary}\label{BettiNumbers}
If the numerical semigroup $\sC$ is obtained by gluing $\sA$ and $\sB$, then
$$
\forall i\geq 0,\
\beta_i(C) =
\sum _{i'=0}^i \beta_{i'}(A)[\beta_{i-i'}(B)+\beta_{i-i'-1}(B)]\,.
$$
\end{corollary}

\begin{proof}
Set $J=I_A R+I_B R$.
Since the resolution of $R/J$ is the tensor product of the resolutions of $\rA$ and $\rB$,
$\displaystyle{
\beta_i(R/J) = \sum_{i'=0}^i \beta_{i'}(A)\beta_{i-i'}(B)
}$
for $i=0,\ldots,p+q-2$.
One can write
$\displaystyle{
\beta_i(R/J) = \sum_{i'\geq 0} \beta_{i'}(A)\beta_{i-i'}(B)
}$ because $\beta_{i-i'}(B)=0$ if $i'>i$, and this formula holds
for all $i\geq 0$ since for $i>p+q-2$, either $i'>p-1$ or $i-i'>q-1$ and hence
$\beta_{i'}(A)\beta_{i-i'}(B)=0$.
Note in particular that
$\beta_{p+q-2}(R/J) =\beta_{p-1}(A)\beta_{q-1}(B)$.
Now the resolution of $\rC$ is the mapping cone of the resolution of $R/J$ induced by multiplication by $\rho$ and hence
$\beta_0(C)=1$,
$\beta_i(C) = \beta_i(R/J)+\beta_{i-1}(R/J)$ for $i=1,\ldots,p+q-2$, and
$\beta_{p+q-1}(C) = \beta_{p+q-2}(R/J)$,
and the result follows.
\end{proof}

\begin{remark}
In Corollary \ref{BettiNumbers}, $A$ and $B$ play the same role and hence the symmetric formula for the Betti numbers also holds:
$$
\forall i\geq 0,\
\beta_i(C) =
\sum _{i'=0}^i \beta_{i'}(B)[\beta_{i-i'}(A)+\beta_{i-i'-1}(A)]\,.
$$
\end{remark}

\begin{example}\label{ex2}
Consider the decomposable numerical semigroup $\sC$ generated by
$$
C=\{450,522,612,576,305,793,732\}=k_1 A\sqcup k_2 B
$$
with $A=\{25,29,34,32\}$, $B=\{5,13,12\}$, $k_1=18=5+13\in\sB$ and $k_2=61=29+32\in\sA$.
Using {\tt Singular}, one gets that the Betti numbers of $\rA$ and $\rB$ are
$$
\begin{array}{|c|c|c|c|c|}
 \hline
i&0&1&2&3\\ \hline
\beta_i(A)&1&7&10&4 \\ \hline
\end{array}
\qquad\qquad
\begin{array}{|c|c|c|c|}
 \hline
i&0&1&2\\ \hline
\beta_i(B)&1&3&2 \\ \hline
\end{array}
$$
Applying the formula in Corollary~\ref{BettiNumbers} and the remark after the corollary, one gets that $\beta_0(C)=1$, and
$$\begin{array}{rcl}
\beta_1(C)&=&1(3+1)+7.1=1(7+1)+3.1=11,\\
\beta_2(C)&=&1(2+3)+7(3+1)+10.1=1.(10+7)+3(7+1)+2.1=43,\\
\beta_3(C)&=&1.2+7(2+3)+10(3+1)+4.1=1.(4+10)+3(10+7)+2(7+1)=81,\\
\beta_4(C)&=&7.2+10(2+3)+4(3+1)=1.4+3(4+10)+2(10+7)=80,\\
\beta_5(C)&=&10.2+4(2+3)=3.4+2(4+10)=40,\\
\beta_6(C)&=&4.2 = 8,
\end{array}$$
and the minimal free resolution of $\rC$ shows as
$$
0\rar R^8\rar R^{40}\rar R^{80}\rar R^{81}\rar R^{43}\rar R^{11}\rar R\rar \rC\rar 0\,.
$$
\end{example}

As a direct consequence of Corollary \ref{BettiNumbers}, one gets that if $A$ and $B$ have $p$ and $q$ elements respectively,
then the last nonzero Betti number of $\rC$ is
$\beta_{p+q-1}(C) =\beta_{p-1}(A)\beta_{q-1}(B)$. In other words, one has the following result:

\begin{corollary}\label{CMtype}
If the numerical semigroup $\sC$ is obtained by gluing two numerical semigroups $\sA$ and $\sB$,
then the Cohen-Macaulay type of $\rC$ is the product of the Cohen-Macaulay types of $\rA$ and $\rB$.
\end{corollary}

The {\it type} of a numerical semigroup $\sC$ is defined in \cite{Frobergetal} as the number of elements in the set
$C'=\{x\in\Z/\,x\notin \sC\hbox{ and }x+s\in \sC, \forall s\in \sC\setminus\{0\}\}$.
In \cite[Prop.~6.6]{nari} it is shown that the type of a numerical semigroup obtained by gluing two semigroups is the
product of the types of the two semigroups. In order to see that this is the same result as Corollary \ref{CMtype},
we need to show the following result that we will prove since we could not find any reference for it.

\begin{lemma}\label{typeCMtype}
The type of any numerical semigroup $\sC$
coincides with the Cohen-Macaulay type of its semigroup ring $\rC$.
\end{lemma}

\begin{proof}
The Cohen-Macaulay type of any simplicial Cohen-Macaulay semigroup is computed in \cite[Thm.~4.2(ii)]{CG}. Applying this result
to a numerical semigroup $\sC$ minimally generated by $C$,
if $e$ is the smallest element in $C$ and $A=C\setminus\{e\}$, the Cohen-Macaulay type of $\rC$ is the number of
elements in the set $Q=\{q\in\sC/\, q-e\notin \sC\}$ that are maximal in the sence that $q+a\notin Q$ for all $a\in A$.
Now one can easily check that $q\in Q$ is maximal if and only if $q-e\in S'$, and the result follows.
\end{proof}

The following corollary gives more direct consequences of Corollaries \ref{BettiNumbers} and \ref{CMtype}:

\begin{corollary}\label{Gorenstein}
Let $\sC$ be a numerical semigroup obtained by gluing $\sA$ and $\sB$.
\begin{enumerate}
\item\label{gorItem}
$\rC$ is Gorenstein, respectively a complete intersection, if and only if
$\rA$ and $\rB$
are both Gorenstein, respectively complete intersections.
\item
If neither
$\rA$
nor
$\rB$
is Gorenstein, then the Cohen-Macaulay type of
$\rC$
is not prime.
\end{enumerate}
\end{corollary}

\begin{remark}
The result for complete intersections is stronger than what is stated in Corollary \ref{Gorenstein} (\ref{gorItem}): every complete intersection
numerical semigroup is indeed the gluing of two complete intersection subsemigroups.
This was shown in \cite{rosales} and previously in \cite{De}; see \cite[Thm. 9.10]{sgBook}.
Note that this result does not hold for Gorenstein semigroups since there exist indecomposable Gorenstein semigroups
when the embedding dimension is $\geq 4$ as we will see later in example \ref{ex5}.
\end{remark}

One can be more precise and compute the graded Betti numbers of a decomposable numerical semigroup.
For any graded $R$-module $M$, we denote by $\beta_{ij}(M)$ the number of syzygies of degree $j$ at the $i$-th step
of a minimal graded free resolution. The set of integers $\{\beta_{ij}(M)\}$ are the {\it graded Betti numbers} of $M$. As for global Betti numbers,
for the semigroup ring $\rC$ associated to a numerical semigroup $\sC$, we will use the notation $\beta_{ij}(C)=\beta_{ij}(\rC)$.

\begin{proposition}\label{gradedBettiNumbers}
If the numerical semigroup $\sC$ decomposes as $C = k_1A\sqcup k_2B$, then
$$
\beta_{i,j}(C)=
\sum_{i'=0}^{i}\Bigg (
\sum_{\tiny
r,s/k_1r+k_2s =j
}
\beta_{i'r}(A)[ \beta_{i-i',s}(B)+\beta_{i-i'-1,s-k_1}(B)]\Bigg )\,.
$$
\end{proposition}

\begin{proof}
The graded Betti numbers of $\langle k_1A\rangle$ and $\langle k_2B\rangle$ are
$\beta_{i,k_1j}(k_1A)=\beta_{ij}(A)$ and
$\beta_{i,k_2j}(k_2B)=\beta_{ij}(B)$.
Set $J=I_A R+I_B R$.
The resolution  of $R/J$ is  $F_A\otimes F_B$ and hence, at the $i$-th step of the
resolution, for $1\leq i\leq p+q-2$, there will be a generator of degree $k_1r+k_2s$
any time $\beta_{i',r}(A)\neq 0$ and $\beta_{i-i',s}(B)\neq 0$ for some $i'$, $0\leq i'\leq i$.
In other words, if we set
$t_i(A) =\{j\in\N\,/\ \beta_{ij}(A)\neq 0\}$ for all $i\in\{0,\ldots,p-1\}$ and $t_i(B) =\{j\in\N\,/\ \beta_{ij}(B)\neq 0\}$ for all $i\in\{0,\ldots,q-1\}$,
the minimal free resolution of $R/J$ is
$$
0\to F_{p+q-2} \to \cdots F_i \to F_{i-1} \to\cdots\to F_1\to  R\to R/J\to 0
$$
with
$\displaystyle{
F_i=\bigoplus_{i'=0}^{i} \Big(
\bigoplus_{\substack{r\in t_{i'}(A)\\ s\in t_{i-i'}(B)}}
R(-k_1r-k_2s)^{\beta_{i'r}(A) \beta_{i-i',s}(B)}\Big)
}$.
Thus, the graded Betti numbers of $R/J$ are, for $0\leq i\leq p+q-2$,
$$
\beta_{i,j}(R/J)=
\sum_{i'=0}^{i}\Bigg(
\sum_{\substack{ r\in t_{i'}(A),s\in t_{i-i'}(B)/ \\  j=k_1r+k_2s  }}
\beta_{i'r}(A) \beta_{i-i',s}(B)\Bigg)\,.
$$
Now the resolution of $\rC$ is the mapping cone of $\rho: F_A\otimes F_B \to F_A\otimes F_B$ induced by multiplication by $\rho $
which has degree $k_1k_2$.
Hence the graded Betti numbers of $\rC$ are $\beta_{00}(C)=1$,
$\beta_{ij}(C) = \beta_{ij}(R/J)+\beta _{i-1,j-k_1k_2}(R/J)$ for $1\leq i\leq p+q-2$,
and
$\displaystyle{\beta_{p+q-1,j}(C)=
\beta _{p+q-2,j-k_1k_2}(R/J)
}$, and we have, for all $i,j\geq 0$,\\
{\footnotesize
\begin{eqnarray*}
\beta_{ij}(C) &=&
\sum_{i'=0}^{i}\Bigg (
\sum_{
r,s/  j=k_1r+k_2s
}
\beta_{i'r}(A) \beta_{i-i',s}(B)\Bigg )+
\sum_{i'=0}^{i-1}\Bigg (
\sum_{
r,s/  j=k_1r+k_2s+k_1k_2
}
\beta_{i'r}(A) \beta_{i-i'-1,s}(B)\Bigg )\\
&=&
\sum_{i'=0}^{i}\Bigg (
\sum_{
r,s/  j=k_1r+k_2s
}
\beta_{i'r}(A) \beta_{i-i',s}(B)\Bigg )+
\sum_{i'=0}^{i}\Bigg (
\sum_{
r,s/  j=k_1r+k_2s
}
\beta_{i'r}(A) \beta_{i-i'-1,s-k_1}(B)\Bigg )
\end{eqnarray*}}\\
because in the second summand, for $i'=i$, one has $i-i'-1<0$, and the result follows.
\end{proof}

\begin{remark}
As for Corollary \ref{BettiNumbers}, the formula obtained by exchanging $A$ and $B$ in the formula in Proposition \ref{gradedBettiNumbers} also holds.
\end{remark}

For any graded $R$-module $M$, the {\it Castelnuovo-Mumford regularity} (or {\it regularity} for short) of $M$, $\reg (M)$, is
$$
\reg (M)=\max\{j-i, i\geq 0,j\geq 0\,/\beta_{ij}(M)\neq 0\}\,.
$$
Again,
for the semigroup ring $\rC$ associated to a numerical semigroup $\sC$, we will use the notation $\reg(C)=\reg(\rC)$.
In this context, the regularity is sometimes called the {\it weighted regularity} since we give weights to the variables to make our
semigroup rings graded modules over the corresponding polynomial ring.

\begin{corollary}{\label{regularity}}
If the numerical semigroup $\sC$ is the gluing of two numerical semigroups $\sA$ and $\sB$
with $C= k_1A\sqcup k_2B$, $\# A=p$ and  $\# B=q$, then
$$\reg(C)=k_1\reg(A)+k_2\reg(B)+(p-1)(k_1-1)+(q-1)(k_2-1)+k_1k_2-1\,.$$
\end{corollary}

\begin{proof}
Since $\rA$, $\rB$ and $\rC$ are Cohen-Macaulay, their regularity is attained at the last step of their minimal graded free resolutions, i.e.
$\reg(A)=\delta_A-(p-1)$,
$\reg(B)=\delta_B-(q-1)$,
and
$\reg(C)=\delta_C-(p+q-1)$ where $\delta_A$, $\delta_B$ and $\delta_C$ is the maximal degree of a syzygy in the last step of a minimal graded free resolution
of $\rA$, $\rB$ and $\rC$ respectively.
According to the argument in the proof of Proposition \ref{gradedBettiNumbers}, the maximal degree of a syzygy in the
last step of a minimal graded free resolution of $R/I_AR+I_BR$ is $k_1\delta_A+k_2\delta_B$, and hence
$\delta_C=k_1\delta_A+k_2\delta_B+k_1k_2$,
and the result follows.
\end{proof}

\begin{example}
Computing the degrees of all the syzygies in the resolutions of $\rA$ and $\rB$ in example \ref{ex1}, one gets that $\reg(A)=134$ and $\reg(B)=49$.
The formula in Corollary \ref{regularity} then gives that $\reg(C)=4227$ which can be checked using {\tt Singular}.
\end{example}

For a numerical semigroup $\sC$, denote by $H_C (t)$ the Hilbert series
of the ring $\rC$.
Using the well-known relation between graded Betti numbers and Hilbert series (see, e.g., \cite[Thm. 1.11]{eis}),
we recover the formula in \cite[Corollary 16]{AGO}.

\begin{corollary}{\label{Hilbert series}}
If the numerical semigroup $\sC$ is the gluing of two numerical semigroups $\sA$ and $\sB$
with $C= k_1A\sqcup k_2B$,
then
$H_C (t)= (1-t^{k_1k_2})H_A(t^{k_1})H_B(t^{k_2})$.
\end{corollary}

\begin{proof}
Set $J=I_A R+I_B R$.
The Hilbert series of $R/J$ is
$$
H_{R/J}(t)=
\frac{\sum_{j}(\sum_{i=0}^{p+q-2}(-1)^i\beta_{ij}(R/J)t^j)}{\prod_{i=1}^p(1-t^{k_1a_i})\prod _{i=1}^q(1-t^{k_2b_i})}\,.
$$
The numerator of this series, denoted by $N_{R/J}(t)$, can be computed using
the formula for the graded Betti numbers of $R/J$ obtained in the proof of Proposition \ref{gradedBettiNumbers}.
Using the sets $t_i(A)$ and  $t_i(B)$ defined there,
one has
\begin{eqnarray*}
N_{R/J}(t)&=&
\sum_{i=0}^{p+q-2}(-1)^i\Big(
\sum_{i'=0}^{i}
\sum_{\substack{r\in t_{i'}(A)\\ s\in t_{i-i'}(B)}}
\beta_{i'r}(A) \beta_{i-i',s}(B)t^{k_1r+k_2s}
\Big)
\\
&=&
\sum_{i=0}^{p+q-2}
\sum_{i'=0}^{i}\Big(
\sum_{\substack{r\in t_{i'}(A)\\ s\in t_{i-i'}(B)}}
(-1)^{i'}\beta_{i',r}(A) t^{k_1r}\times (-1)^{i-i'}\beta_{i-i',s}(B)t^{k_2s}
\Big)
\\
&=&
\Big(
\sum_{i=0}^{p-1}(-1)^i
\sum_{r\in t_{i}(A)}
\beta_{ir}(A) t^{k_1r}
\Big)
\Big(
\sum_{i=0}^{q-1}(-1)^i
\sum_{s\in t_{i}(B)}
\beta_{is}(B)
t^{k_2s}
\Big)
\end{eqnarray*}
and hence $H_{R/J}(t)=H_A(k_1t)\times H_B(k_2t)$.
On the other hand, the Hilbert Series of $\rC$ is
$$
H_C(t)=
\frac{
\sum_{j}(
\sum_{i=0}^{p+q-1}(-1)^i\beta_{ij}(C)t^j
)
}{
\prod_{i=1}^p(1-t^{k_1a_i})\prod _{i=1}^q(1-t^{k_2b_i})
}
$$
and, since for all $i,j\geq 0$,
$\beta_{ij}(C) = \beta_{ij}(R/J)+\beta _{i-1,j-k_1k_2}(R/J)$,
the numerator of this series, $N_{C}(t)$, is
\begin{eqnarray*}
N_C(t)&=&
\sum_{j}\Big(
\sum_{i=0}^{p+q-2}(-1)^i\beta_{ij}(R/J)t^j\Big)+
\sum_{j}\Big(
\sum_{i=1}^{p+q-1}(-1)^i\beta_{i-1,j-k_1k_2}(R/J)t^j\Big)
\\
&=&
\sum_{j}\Big(
\sum_{i=0}^{p+q-2}(-1)^i\beta_{ij}(R/J)t^j\Big)+
\sum_{j}\Big(
\sum_{i=0}^{p+q-2}(-1)^{i+1}\beta_{ij}(R/J)t^{j+k_1k_2}\Big)
\\
&=&
(1-t^{k_1k_2})N_{R/J}(t)\,,
\end{eqnarray*}
and hence $H_C(t)=(1-t^{k_1k_2})H_{R/J}(t)$, and the result follows.
\end{proof}

One can also show that
the minimal free resolution of $\rC$ constructed in Theorem \ref{main}
inherits the DG algebra structure of $\rA$ and $\rB$.

\medskip
A complex of free $R$-modules,
${\bf F}:  \ldots \to F_i \stackrel {d_i} \to F_{i-1} \ldots \to F_1 \to F_0 = R$ is said to have a {\it differential graded algebra structure}
({\it DG algebra structure} for short)
if there is a multiplication $*$ on $F =\oplus F_i$ which makes it an associative and graded commutative differential graded algebra; see \cite{BE2}.
This means that $*$ is associative and, for all $a\in F_i$ and $ b\in F_j$, $a*b \in F_{i+j}$, $b*a= (-1)^{i+j}a*b$, and $d_{i+j} (a*b) = d_i(a)*b+(-1)^ia*d_j(b)$.
This last equality is called the Leibnitz rule.
When a resolution of an $R$-algebra has a DG algebra structure, we say that it is a {\it DG algebra resolution}.
For a detailed study of these multiplicative structures on resolutions, see, e.g., \cite{Av}.

\medskip
In our case, using the construction in Theorem \ref{main}, we can write the multiplicative structure explicitly.

\begin{proposition}\label{DGstructure}
Suppose that the numerical semigroup $\sC$ is obtained by gluing $\sA$ and $\sB$.
If the semigroup rings $\rA$ and $\rB$ have minimal DG algebra resolutions, then so does the semigroup ring $\rC$.
\end{proposition}

\begin{proof}
Assume that $F_A$ and $F_B$ in section \ref{notationsGluingSec}  have a DG algebra structure over $R_A$ and $R_B$ respectively.
Then $F_A\otimes F_B$, which is a resolution of $R/I_A R+I_B R$, admits the following DG algebra structure:
let $a_1\otimes b_1 \in {(F_A)}_{i}\otimes {(F_B)}_{r-i}$ and $a_2\otimes b_2 \in {(F_A)}_j\otimes {(F_B)}_{s-j}$ be two basis elements in
${(F_A\otimes F_B)}_r$ and ${(F_A\otimes F_B)}_s$ respectively,
and consider the following product:
$$(a_1\otimes b_1 ) (a_2\otimes b_2) = (-1)^{(r-i)j}(a_1a_2)\otimes (b_1b_2)\,.$$
One has that
\begin{eqnarray*}
(a_2\otimes b_2) (a_1\otimes b_1)
&=& (-1)^{(s-j)i+ij+(r-i)(s-j)}(a_1a_2)\otimes (b_1b_2) \\
&=&(-1)^{rs+(r-i)j}a_1a_2\otimes b_1b_2 = (-1)^{rs} (a_1\otimes b_1 ) (a_2\otimes b_2)\,.
\end{eqnarray*}
Further, for $w=(a_1\otimes b_1 ) (a_2\otimes b_2)$, if $\partial$ denotes the differential, one has that
\begin{eqnarray*}
\partial w&=&
(-1)^{(r-i)j}[(\partial a_1 a_2+(-1)^i a_1\partial a_2)\otimes b_1b_2+
\\
&& \hskip 5.3cm
(-1)^{i+j}a_1a_2\otimes (\partial b_1b_2+(-1)^{r-i} b_1\partial b_2)]
\\
&=&
(\partial a_1\otimes b_1)(a_2\otimes b_2)+(-1)^i (a_1\otimes \partial b_1)(a_2\otimes b_2)+
\\
&& \hskip 2.5cm (-1)^r(a_1\otimes b_1)(\partial a_2\otimes b_2)+(-1)^{r+j} (a_1\otimes b_1)(a_2\otimes \partial b_2)
\\
&=&
(\partial (a_1\otimes b_1))(a_2\otimes b_2)+(-1)^r(a_1\otimes b_1)(\partial (a_2\otimes b_2))
\end{eqnarray*}
and hence it satisfies the Leibnitz rule.
To check associativity, all we need to check is the sign.
As the equality $\deg b_1 \deg a_2+(\deg b_1+\deg b_2)\deg a_3= \deg b_1 (\deg a_2+\deg a_3)+\deg b_2 \deg a_3$ always holds,
we are done and hence $F_A\otimes F_B$ is a DG algebra resolution of $R/I_A R+I_B R$ with this product.

\medskip
Denote the complex $F_A\otimes F_B$ by $G$ for convenience.
We have a complex map $\rho:G \to G$ given by multiplication by $\rho$.  Then the mapping cone
of $\rho$ has a multiplicative structure
$\star:G_i\oplus G_{i-1}\times G_j\oplus G_{j-1} \to G_{i+j}\oplus G_{i+j-1}$ given by
$$(a_1,b_1)\star (a_2, b_2) = ( a_1 a_2, a_1b_2+(-1)^jb_1a_2)\,.$$
For the sake of completion, we verify that this provides a multiplicative structure.
\begin{eqnarray*}
(a_2,b_2)\star(a_1,b_1)& =& ( a_2a_1,a_2b_1+(-1)^i b_2a_1) = ( -1)^{ij} (a_1a_2, a_1b_2+(-1)^jb_1a_2)\\
&=& (-1)^{ij} (a_1,b_1)\star (a_2,b_2)
\end{eqnarray*}
so it is graded commutative.
It is associative because
$(a_1,b_1)\star(a_2,b_2)\star(a_3,b_3) = ( a_1a_2a_3, a_1a_2b_3+(-1)^{k}a_1b_2a_3+(-1)^{(j+k)}b_1a_2a_3 )$.

\medskip
Finally, to check the Leibnitz rule, we see that for $w=( a_1, b_1)\star(a_2,b_2)$, denoting again the differential by $\partial$,
we have
\begin{eqnarray*}
\partial w&=&
(\partial (a_1a_2)+(-1)^{i+j-1}\rho (a_1b_2+(-1)^jb_1a_2), \partial (a_1b_2+(-1)^jb_1a_2))
\\
&=&
(\partial (a_1) a_2+(-1)^ia_1 \partial (a_2)+(-1)^{i+j-1}\rho (a_1b_2+(-1)^jb_1a_2),
\\
&&\hskip 1.5cm
\partial(a_1)b_2+(-1)^i a_1\partial (b_2)+
(-1)^j(  \partial (b_1)a_2+(-1)^{i-1}b_1\partial (a_2) ))
\\
&=&
(\partial (a_1) +(-1)^{i-1}\rho b_1, \partial (b_1))\star (a_2, b_2)+
\\
&&\hskip 4.5cm
(-1)^i (a_1,b_1)\star (\partial(a_2)+(-1)^{j-1}\rho b_2, \partial (b_2))\,.
\end{eqnarray*}
This completes the proof.
\end{proof}

\begin{remark}
In fact, the following general result is true: if $C_1$ and $C_2$ are two complexes with
a DG algebra structure, if $C_2$ has a structure of $C_1$-module, and if $\phi:C_1\to C_2$ is a complex map which is also a $C_2 $-module homomorphism, then the mapping cone of $\phi$ has a structure of DG algebra.
This can be applied at the end of the proof of Theorem~\ref{DGstructure} to $C_1 = C_2 = F_A\otimes F_B$ since the map $\rho$ is clearly a module homomorphism because it is the multiplication by an element of the ring $R$. We gave a direct proof of the specific result above for completion.
\end{remark}

\section{Simple split}
\label{simplesplitSec}

Let's focus here on a special class of decomposable numerical semigroups:
if the numerical semigroup $\sC$ decomposes as $C = k_1\{1\}\sqcup k_2A$ where $k_1\in \sA\setminus A$, $k_2\neq 1$ and $\gcd(k_1,k_2) = 1$, we
call the decomposition a  {\it simple split} and denote $C= k_1\sqcup k_2A$ for simplicity.
The structure of the resolution of $\rC$ and the value of its Betti numbers come as a direct consequence of Theorem \ref{main} and
Corollary \ref{BettiNumbers},
but it can be obtained, in this particular case, by a simple and direct argument.

\begin{theorem}\label{simpleSplit}
Suppose that $C= k_1\sqcup k_2A$ is a simple split.  Then, a minimal free resolution of the semigroup ring $\rC$ is obtained as the mapping cone over
a minimal resolution of $\rA$ induced by multiplication by a single element.
In particular, $\beta_i(C) = \beta_i(A)+\beta_{i-1}(A)$, and the Cohen-Macaulay type of $\rC$ is the same as the type of $\rA$.
\end{theorem}

\begin{proof}
Assume that the embedding dimension of $A$ is $n$, $A=\{a_1,\ldots,a_n\}$, and denote $R=k[x_0,\ldots,x_n]$.
The following diagram commutes:
$$
\begin{array}{ccccccccc}
 &    &0       &    &0                &    &0                                      &    & \\
 &    &\uparrow&    &\uparrow         &    &\uparrow                               &    & \\
0&\rar&I_C/I_AR     &\rar&k[x_0]           &\rar&\coker{(\delta)}                       &\rar&0\\
 &    &\uparrow&    &\uparrow         &    &\uparrow                               &    & \\
0&\rar&I_C       &\rar&k[x_0,\ldots,x_n]&\rar&k[t^{k_1},t^{k_2a_1},\ldots,t^{k_2a_n}]&\rar&0\\
 &    &\uparrow&    &\uparrow         &    &\uparrow  \delta                       &    & \\
0&\rar&I_AR       &\rar&k[x_1,\ldots,x_n]&\rar&k[t^{k_2a_1},\ldots,t^{k_2a_n}]        &\rar&0\\
 &    &\uparrow&    &\uparrow         &    &\uparrow                               &    & \\
 &    &0       &    &0                &    &0                                      &    & \\
\end{array}
$$
The first row of the diagram makes $I_C/I_AR$ an ideal of $k[x_0]$ which is hence principal: it is generated by the class of the element
$\rho= x_0^{k_1}-x_1^{\alpha_1}\ldots x_n^{\alpha_n}$
coming from the fact that $k_1\in\sA$ and so $k_1=\alpha_1 a_1+\cdots+\alpha_n a_n$.
Let $F_A:\ 0\to F_n \stackrel{\partial_n} \to F_{n-1} \to \ldots F_i\stackrel{\partial_i}\to F_{i-1} \to \ldots F_1\to R$ be a minimal resolution of $R/I_A R$ and consider
the complex map $\rho: F_A\to F_A$ given by multiplication by $\rho$ .
Now, $R/I_A R$ is Cohen Macaulay so all the maps $\partial_i: F_i \to F_{i-1} $ in the resolution have Fitting ideals equal up to radical to $I_A$ and hence have the same depth $n$.  This follows from \cite[Cor. 2.5]{S}.
The $i $th Fitting ideal of the mapping cone is $I(\partial_i)I(\partial_{i-1})+(\rho)$   and  has  depth $n+1$ since $\rho\notin  \sqrt (I_A)$.   So, the mapping cone of the complex map $\rho$ is exact and it resolves $R/I_A R+(\rho) = R/I_C$.
\end{proof}

These simple splits will appear in the next section when we describe decomposable numerical semigroups with small embedding dimension.
One can also use simple splits to construct easily complete intersections of higher embedding dimension by adding generators in the semigroup one by one
as the following example shows.

\begin{example}\label{ex3}
The semigroup $\sA$ minimally generated by the simple split $A=\{11\}\sqcup 2\{3,8\}=\{11,6,16\}$ is a complete intersection
of embedding dimension 3.
One gets a complete intersection semigroup $\sC$ of embedding dimension 4 by considering the following simple split $C=17\sqcup 2A=\{17,22,12,32\}$.
One can iterate the procedure by considering, for example, the simple split $29\sqcup 2C$ to get a complete intersection of embedding dimension 5.
\end{example}

\begin{example}\label{ex4}
Complete intersections with embedding dimension 4 can be simple splits or not. The decomposable semigroup $\sC$ with $C=11\{3,5\}\sqcup 8\{4,7\}=
\{33,55,32,56\}$
is a complete intersection since the ideal $I_{C}$ is generated by the binomial $x_1^{5}-x_2^3$ that generates
the ideal $I_{\langle 3,5\rangle}$, the binomial $x_3^{7}-x_4^4$ that generates
the ideal $I_{\langle 4,7\rangle}$, and the binomial $\rho=x_1x_2-x_3x_4$ but it is not a simple split.
The resolution of $\rC\simeq R/I_{C}$, as the one of $\rC$ in example \ref{ex3}, is a Koszul complex and it has the following form:
$0\rar R\rar R^3\rar R^3\rar R\rar k[C]\rar 0$.
\end{example}

\section{Semigroups with small embedding dimension}
\label{smallembdimSec}

Let $\sC$ be a numerical semigroup minimally generated by $C$.  Then the embedding dimension of the semigroup ring $\sC$ is $n$.
When $n=2$, $\sC$ is a complete intersection, and when $n=3$, then $\sC$ is either indecomposable or a complete intersection by \cite{He}.
In fact, when it is indecomposable of embedding dimension $3$, it is generated by the $2\times 2$ minors of a $2\times 3$ matrix of the form given in
the introduction, just like the semigroup $\sB$ in our examples \ref{ex1} and \ref{ex2}.  This matrix is also the transpose of the matrix
of the first syzygies of the ideal $I_C$. Such ideals are called Hilbert-Burch since they fit in the class of ideals satisfying the theorem of the same name;
see \cite[Thm. 3.2]{eis}.

\medskip
For embedding dimension 4 we have the following:

\begin{proposition}\label{embdim4}
Numerical semigroups of embedding dimension 4 that are decomposable are
either complete intersections or almost complete intersections of type 2.
\end{proposition}

\begin{proof}
When $\sC$ is decomposable of embedding dimension $4$, if $\sC$ is not a simple split then it decomposes
as $C=k_1A\sqcup k_2B$ where both $A$ and $B$ have 2 elements (as in example \ref{ex4}) and $\sC$ is a complete intersection as the gluing of
two complete intersections.
If $\sC$ is a simple split, $k_1\sqcup k_2A$, and $\sA$ is a complete intersection (as in example \ref{ex3}) then $\sC$ is a complete intersection.
If $\sA$ is not a complete intersection, then $I_A$ is generated by the $2\times 2$ minors of a $2\times 3$ matrix and, by Corollary \ref{BettiNumbers},
the minimal free resolution of $\rC$ is as follows: $0\rar R^2\rar R^5\rar R^4\rar R\rar\rC \rar 0$.
\end{proof}

\medskip
At this point one can ask if every numerical semigroup of embedding dimension 4 that is an almost complete intersection of type 2 is decomposable.
The answer is negative as the following example shows: the minimal free resolution of the semigroup $\sC$ minimally generated
by $C=\{9,11,13,15\}$ also has the following form, $0\rar R^2\rar R^5\rar R^4\rar R\rar\rC\rar 0$.
This semigroup is generated by an arithmetic sequence. In \cite{JA}, the minimal graded free resolution of such semigroups has been completely described.
In particular, one gets from \cite[Thm. 4.1]{JA} that for every semigroup $\sC$ of embedding dimension 4 generated by an arithmetic sequence,
if the Cohen-Macaulay type of $\rC$ is 2 then the minimal free resolution of $\rC$ has the following form:
$0\rar R^2\rar R^5\rar R^4\rar R\rar\rC\rar 0$, and hence $\sC$ is an almost complete intersection.
These semigroups are all indecomposable by the following precise result:

\begin{proposition}\label{arithSeq}
A numerical semigroup minimally generated by a set of positive integers in arithmetic progression is
decomposable if and only if it is of embedding dimension 3 and generated by a set of the form $\{2c_0, 2c_0+d, 2c_0+2d\}$
for $c_0$ and $d$ relatively prime and $d$ odd.
In particular, it is indecomposable in embedding dimension 4 and higher.
\end{proposition}

\begin{proof}
If $C$ is an arithmetic sequence and $\sC$ has embedding dimension $n$, then $C = \{ c_0+id, 0\le i\le n-1\}$ where $c_0$  and $d$ are relatively prime.   
If $C$ is decomposable as $k_1A\sqcup k_2B$ , then $c_0$ and $c_0+d$ must be in different sets.  Otherwise, they would have a common factor
which would also be a common factor of $c_0$ and $d$.  Similarly, $c_0+d$ and $c_0+2d$ cannot be in the same set.  Hence $c_0, c_o+2d$ are in the same set, say $k_1A$.  So $k_1$ divides $2d$ and $c_0$, and hence $k_1 = 2$ and $c_0$ is even.
Now, if $n\geq 4$, $c_0+d, c_0+3d$ must be in $k_2B$ and hence $k_2$ divides $2d$ and $c_0+d$.  Since $c_0$ and $d$ are relatively prime, $k_2$ must also be 2.  Then all the integers in $C$ are even which is not possible.  Thus, $C$ is indecomposable if $n\geq 4$.
If $n=3$, $C$ must be of the form $\{2c_0, 2c_0+d, 2c_0+2d\}$ for $c_0$ and $d$ relatively prime and $d$ odd, which is the simple split
$C=(2c_0+d)\sqcup 2\{c_0,c_0+d\}$: this is the only possibility for $\sC$ to be decomposable if the elements in $C$ are in
arithmetic progression.
\end{proof}

Semigroups generated by arithmetic sequences provide examples of indecomposable semigroups of any type between 1 and
the embedding dimension.

\begin{corollary}\label{indecompCMtype}
For any $n\geq 4$ and any $t$, $1\leq t\leq n$, there exist indecomposable numerical semigroups of embedding dimension $n$ and type $t$.
\end{corollary}

\begin{proof}
This is a direct consequence of Proposition \ref{arithSeq} and \cite[Thm. 4.7]{JA}.
\end{proof}

We get, in particular, examples of Gorenstein semigroups that are indecomposable.

\begin{example}\label{ex5}
For all $n\geq 3$, if $C= \{n+1, 2n+1, 3n+1, \ldots, n^2+1\}$ is a set of integers that minimally generates
the numerical semigroup $\sC$, the
embedding dimension of $\sC$ is $n$. For $n=3$, it fits into the family of decomposable semigroups given in Proposition \ref{arithSeq}.
For $n\geq 4$, it is a Gorenstein numerical semigroup by \cite[Cor. 4.4]{JA} and it is indecomposable by Proposition \ref{arithSeq}.
One can obtain all the Betti numbers of $\rC$ using formula (4) in the proof of  \cite[Thm. 4.1]{JA}.
In particular, the ideal $I_C$ is minimally generated by $(n+1)(n-2)/2$ elements.
\end{example}

By Proposition \ref{embdim4}, one gets that semigroups of embedding dimension 4 that are Gorenstein and not complete intersections
are always indecomposable. For such a semigroup $\sC$, Bresinsky showed in \cite{Br} that the ideal $I_C$ is minimally generated
by 5 binomials of the form
$\{x_1^{c_1}-x_3^{d_{13}}x_4^{d_{14}},
x_3^{c_3}-x_1^{d_{31}}x_2^{d_{32}},
x_4^{c_4}-x_2^{d_{42}}x_3^{d_{43}},
x_2^{c_2}-x_1^{d_{21}}x_4^{d_{24}}, x_1^{d_{21}}x_3^{d_{43}}-x_2^{d_{32}}x_4^{d_{14}}\}$.
In \cite{braz}, we give the structure of the minimal resolution of $\rC$ in this case
and explicitly write down the maps in the resolution as follows:
$$
0 \rightarrow R \stackrel{\delta_3}{\rightarrow} R^5 \stackrel{\phi}{\rightarrow} R^{5}\stackrel{\delta_1} \rightarrow R \rightarrow\rC
\rightarrow 0
$$
where $$\phi
 = \left[
\begin{matrix}
0&0& x_2^{d_{32}} &x_3^{d_{43}}&x_4^{d_{24}}\\
0&0&x_1^{d_{21}}&x_4^{d_{14}}&x_2^{d_{42}}\\
-x_2^{d_{32}}&-x_1^{d_{21}}&0&0&x_3^{d_{13}}\\
-x_3^{d_{43}}&-x_4^{d_{14}}&0&0&x_1^{d_{31}}\\
-x_4^{d_{24}}&-x_2^{d_{42}}&-x_3^{d_{13}}&-x_1^{d_{31}}&0
\end{matrix}
\right]\quad\hbox{and}\quad
\delta _3= (\delta _1) ^T=
\left[
\begin{matrix}
x_1^{c_1}-x_3^{d_{13}}x_4^{d_{14}}\\
x_3^{c_3}-x_1^{d_{31}}x_2^{d_{32}}\\
x_4^{c_4}-x_2^{d_{42}}x_3^{d_{43}}\\
x_2^{c_2}-x_1^{d_{21}}x_4^{d_{24}}\\
x_1^{d_{21}}x_3^{d_{43}}-x_2^{d_{32}}x_4^{d_{14}}
\end{matrix}
\right]\,.
$$
This was also obtained independently in \cite{BFS}.
Moreover, one can find in \cite{braz}
a way to build families of numerical semigroups of embedding dimension 4 that are Gorenstein and not complete intersections.

\medskip
For numerical semigroups of embedding dimension 5, one has the following result.

\begin{proposition}\label{embdim5}
In embedding dimension $5$, the decomposable numerical semigroups are either complete intersections or type 2 almost complete intersections or
obtained as a simple split $k_1\sqcup k_2A$ where $A$ is an indecomposable semigroup of embedding dimension $4$.
\end{proposition}

\begin{proof}
If $\sC$ is a simple split $C=k_1\sqcup k_2A$ then either $A$ is indecomposable, or $A$ is a complete intersection, or it is of type 2 and $I_A$ is minimally
generated by 4 elements by Proposition \ref{embdim4}.
In the last case, by Theorem \ref{simpleSplit}, $I_C$ is generated by 5 elements and it is of type 2 and hence it is an almost complete intersection
of type 2.

\medskip
If $\sC$ decomposes as $C=k_1A\sqcup k_2B$ where $A$ and $B$ have embedded dimension 2 and 3 respectively, if $B$ is indecomposable then
by Theorem \ref{main}, $I_C$ is minimally generated by 5 elements and the type of $\sC$ is 2 since it coincides with the type of $\sB$. Again,
in this case $\sC$ is an almost complete intersection of type 2.
If $B$ is decomposable, $\sB$ is a complete intersection and so is $\sC$.
\end{proof}

In particular, one gets the following direct consequence.

\begin{corollary}
In embedding dimension 5, any decomposable Gorenstein numerical semigroup that is not a complete intersection must be obtained as simple split.
\end{corollary}

\section{Semigroups obtained by gluing in higher dimension}
\label{higherdimSec}

One can now consider affine semigroups $\sC$ finitely generated by a subset $C\subset\N^n$. As for numerical semigroups,
given an arbitrary field $k$, if $C=\{\alpha_1,\ldots,\alpha_s\}$, one can consider the ring homomorfism
$\phi: k[x_1, \dots, x_s] \to k[t_1,\ldots, t_n]$ given by $\phi(x_i) = t^{\alpha_i}$ where for $\alpha=(a_1,\ldots,a_n)\in\N^n$,
$t^\alpha = \prod _j t_j^{a_j}$.
The kernel of $\phi$ is a binomial prime ideal, $I_C$, and the semigroup ring $\rC$ is isomorphic to $k[x_1, \dots, x_s]/I_C$.
It is graded if one gives to each variable $x_i$ weight $|\alpha_i|$.

\medskip
In this context, Rosales defined  in \cite{rosales} the concept of gluing. For an affine semigroup $\sC$,
when the set of generators of the semigroup splits into two disjoint parts, $C=A\sqcup B$, such that $I_C$ is minimally generated
by $I_A\cup I_B \cup\{\rho\}$ where $\rho$ is a binomial whose first, respectively second, monomial involves only variables corresponding to elements
in $A$, respectively $B$, we say that $\sC$ is obtained by {\it gluing} $\sA$ and $\sB$.
A characterization of this property in terms of the semigroups $\sA$ and $\sB$ and the groups associated to them is given
in \cite[Thm. 1.4]{rosales}; see also \cite{AGO}.

\medskip
In this situation, the argument in the proof of Theorem \ref{main} is valid and hence Theorem~\ref{main} and Corollaries
\ref{BettiNumbers} and \ref{CMtype} hold. We will use again notations in section \ref{notationsGluingSec}
except the ones involving $k_1$ and $k_2$ that are not defined in this general context.

\begin{theorem}\label{mainHD}
Consider an affine semigroup $\sC$ obtained by gluing $\sA$ and $\sB$.
\begin{enumerate}
\item\label{mainHDidealJ}
$F_A\otimes F_B$ is a minimal graded free resolution of $R/I_A R+I_B R$.
\item\label{mainHDsgC}
A minimal graded free resolution of the semigroup ring $\rC$ can be obtained as the mapping cone
of $\rho: F_A\otimes F_B\to F_A\otimes F_B$ where $\rho$ is induced by multiplication by $\rho$.
\item\label{BettiNumbersHD}
The Betti numbers of $\rC$ are
$$
\forall i\geq 0,\
\beta_i(C) =
\sum _{i'=0}^i \beta_{i'}(A)[\beta_{i-i'}(B)+\beta_{i-i'-1}(B)]=
\sum _{i'=0}^i \beta_{i'}(B)[\beta_{i-i'}(A)+\beta_{i-i'-1}(A)]\,.
$$
\item\label{CMtypeHD}
In particular,
the Cohen-Macaulay type of $\rC$ is the product of the Cohen-Macaulay types of $\rA$ and $\rB$.
\end{enumerate}
\end{theorem}

\begin{proof}
(\ref{mainHDidealJ}) and (\ref{mainHDsgC}) are obtained exactly as Theorem \ref{main} since
our proof depends only on the form of the generating set of $I_C$.
On the other hand, the proof of Corollary \ref{BettiNumbers} also works in this general context by substituting $p-1$ and $q-1$ by $\pd(\rA)$ and $\pd(\rB)$,
the projective dimensions of $\rA$ and $\rB$, and it gives (\ref{BettiNumbersHD}).
(\ref{CMtypeHD}) is a direct consequence of (\ref{BettiNumbersHD}).
\end{proof}

A main difference between that case of numerical semigroups and the general case is that semigroup rings associated to affine semigroups are not
always Cohen-Macaulay. The relation between the lengths of the minimal graded free resolutions of $\rA$, $\rB$ and $\rC$ is given
by the following direct consequence of Theorem \ref{mainHD}(\ref{mainHDsgC}):

\begin{corollary}\label{pdHD}
If the affine semigroup $\sC$ is obtained by gluing $\sA$ and $\sB$, then
$\pd(\rC)=\pd(\rA)+\pd(\rB)+1$.
\end{corollary}

\begin{example}
Consider the subsets of $\N^3$, $A=\{(4,0,0),(3,1,0),(2,2,0),(1,3,0)\}$, $B=\{(3,3,0),(3,2,1),(3,1,2),(3,0,3)\}$
and $C=A\sqcup B$, and the affine semigroups $\sA$, $\sB$ and $\sC$.
Set $R_A=k[x_1,\ldots,x_4]$, $R_B=k[y_1,\ldots,y_4]$ and $R=k[x_1,\ldots,y_4]$.
Using {\tt Singular} \cite{singular}, one gets that the ideals $I_A\subset R_A$ and $I_B\subset R_B$ are minimally generated by
$\{x_3^2-x_2x_4,x_2^2-x_1x_3,x_2x_3-x_1x_4\}$ and $\{y_3^2-y_2y_4,y_2^2-y_1y_3,y_2y_3-y_1y_4\}$,
and the ideal $I_C\subset R$ is minimally generated by the union of those two sets and the element $\rho=y_1^2-x_1x_4^4$.
The minimal free resolutions of the semigroup rings $\rA$ and $\rB$, as $R$-modules, are of the following form,
$0\rar R^2\rar R^3\rar R$ and, by Theorem \ref{mainHD}, that of $\rC$ is $0\rar R^4\rar R^{16}\rar R^{25}\rar R^{19}\rar R^7\rar R$.
\end{example}

\section*{Acknowledgements}
The second author acknowleges with pleasure the support and hospitality of University of Valladolid while working on this project.

\end{document}